\documentclass[11pt]{amsart}
\usepackage{hyperref,mathrsfs}
\usepackage{color,graphicx} 
\usepackage{amssymb,latexsym,amsmath}     % Standard packages

\newtheorem{theorem}{Theorem}[section]
\newtheorem{proposition}{Proposition}[section]

\newtheorem{corollary}{Corollary}[section]
\begin{document}
\title[Volume estimates for right-angled polyhedra]{Volume estimates for right-angled hyperbolic polyhedra}
\author{A. Egorov, A. Vesnin}
\address{Novosibirsk State University, Novosibirsk, 630090, Russia} 
\email{a.egorov2@g.nsu.ru }
\address{Tomsk State University, Tomsk, 634050, Russia} 
\email{vesnin@math.nsc.ru}

\begin{abstract} 
By Andreev  theorem acute-angled polyhedra of finite vo\-lume in a hyperbolic space $\mathbb H^{3}$ are uniquely determined by combinatorics of their 1-skeletons and dihedral angles. For a class of compact right-angled polyhedra and a class of ideal right-angled polyhedra estimates of volumes in terms of the number of vertices were obtained by Atkinson in 2009. In the present paper upper estimates for both classes are improved.
\end{abstract}
\dedicatory{
To Bruno Zimmermann on his 70-th birthday
} 
\thanks{The work was supported by the Theoretical Physics and Mathematics Advancement Foundation ``BASIS''} 
	
\maketitle 

%%%%%%

	\section{Introduction}
	
	In the present paper we consider polyhedra of finite volume in a hyperbolic 3-space $\mathbb H^{3}$. 
	A  hyperbolic polyhedron is said to be \emph{right-angled} if all its dihedral angles are equal to $\pi/2$. 
	Necessary and sufficient conditions for realization in $\mathbb H^{3}$ of a polyhedron with given combinatorial type and dihedral angles were described by Andreev~\cite{A70}. Moreover, the realization is unique up to isometry.  A~hyperbolic polyhedron is said to be  \emph{ideal} if all its vertices belong to $\partial \mathbb H^{3}$. 
	The smallest 825 compact right-angled hyperbolic polyhedra were determined by Inoue~\cite{I20}. In~\cite{EV20} we listed 248 initial volumes of ideal right-angled hyperbolic polyhedra and formulated a conjecture about smallest volume polyhedra when number of vertices is fixed. 
	
	Right-angled polyhedra (both compact and ideal) are very useful building-blocks for construction of hyperbolic 3-manifolds with interesting properties~\cite{V87, V17}. In particular, the right-angled decomposition  gives an immersed totally geodesic surface in the 3-manifold arising from the faces of the polyhedra. 
	Recently nice combinatorial characterization of right-angled hyperbolic 3-orbifolds was obtained in~\cite{Lu20}. 
	
	It is known that complements of the Whitehead link and of the Borromean rings admit decompositions into ideal right-angled octahedra. 
	Following~\cite{CKP19}, a~hyperbolic link in 3-sphere is said to be right-angled if its complement equipped with  the complete hyperbolic structure admits a decomposition into ideal hyperbolic right-angled polyhedra.  Among non-alternating links, the class of fully augmented links is right-angled (see recent results on augmented links in~\cite{F18, MMT20}). 
	The following conjecture was proposed in~\cite{CKP19}: there does not exist a right-angled knot. In the same paper the conjecture was verified for knots up to $11$ crossings, basing on the tabulation of volumes of ideal right-angled polyhedra given in~\cite{EV20}.

	\section{Main results}
	
	Since right-angled hyperbolic polyhedra are determined by combinatorial structure, it is natural to expect that their geometrical invariants, e.g. volume, can be estimated by combinatorial invariants, e.g. number of vertices. 
	
	Two-sided volume estimates for ideal right-angled hyperbolic polyhedra in terms of the number of vertices were obtained by Atkinson~\cite{A09}. 
	
	\begin{theorem} {\rm \cite[Theorem~2.2.]{A09}} \label{th101} 
		If $P$ is an ideal right-angled hyperbolic polyhedron with  $V$ vertices, then 
		$$
		(V-2) \cdot \frac{v_{8}}{4} \leqslant \operatorname{vol} (P) \leqslant (V-4) \cdot \frac{v_{8}}{2},  
		$$
		where $v_{8}$ is the volume of a regular ideal hyperbolic octahedron. 
		Both inequalities are equalities when $P$ is the regular ideal hyperbolic octahedron. There is a sequence of ideal $\pi/2$-equiangular polyhedra $P_{i}$ with $V_{i}$ vertices such that $\operatorname{vol} (P_{i}) / V_{i}$ approaches $v_{8} / 2$ as $i$ goes to infinity. 
	\end{theorem}
	
	Constant $v_{8}$ has an expression in terms of the Lobachevsky function, 
	$$
	\Lambda (x) = - \int_{0}^{x} \log | 2 \sin t | dt, 
	$$
	which dates back to N.\,I.~Lobachevsky's 1832 paper.  Namely, $v_{8} = 8 \Lambda (\pi/4)$. 
	To fifteen decimal places,  $v_{8}$ is $3.663862376708876$.
	
	Recall that if $P \subset \mathbb H^{3}$ is ideal right-angled, then $V \geqslant 6$. The case  $V = 6$ is realized for an octahedron that is an antiprism with a triangular top and bottom, the case $V = 8$ is realized for an antiprism with quadrilateral top and bottom, and there no ideal right-angled hyperbolic polyhedra with $V=7$ vertices. 
	The upper bound from Theorem~\ref {th101} can be improved if we exclude from consideration the two smallest ideal right-angled hyperbolic polyhedra. 
	
	\begin{theorem} \label{th102} 
		If $P$ is an ideal right-angled hyperbolic polyhedron with $V \geqslant 9$ vertices, then  
		$$
		\operatorname{vol} (P) \leqslant (V-5) \cdot \frac{v_{8}}{2}.
		$$ 
		The equality holds if and only if $V=9$. 
	\end{theorem}
	
	Proof of Theorem~\ref{th102} is given in Section~\ref{sec2}. In Figure~\ref{fig1} dots present volumes of ideal right-angled hyperbolic polyhedra with at most $21$ vertices, calculated in~\cite{EV20}, and lines present volume estimates from Theorem~\ref{th101} (lower bound) and Theorem~\ref{th102} (upper bound).
	\begin{figure}[ht]
		\centering
		\includegraphics[width=1.0\textwidth]{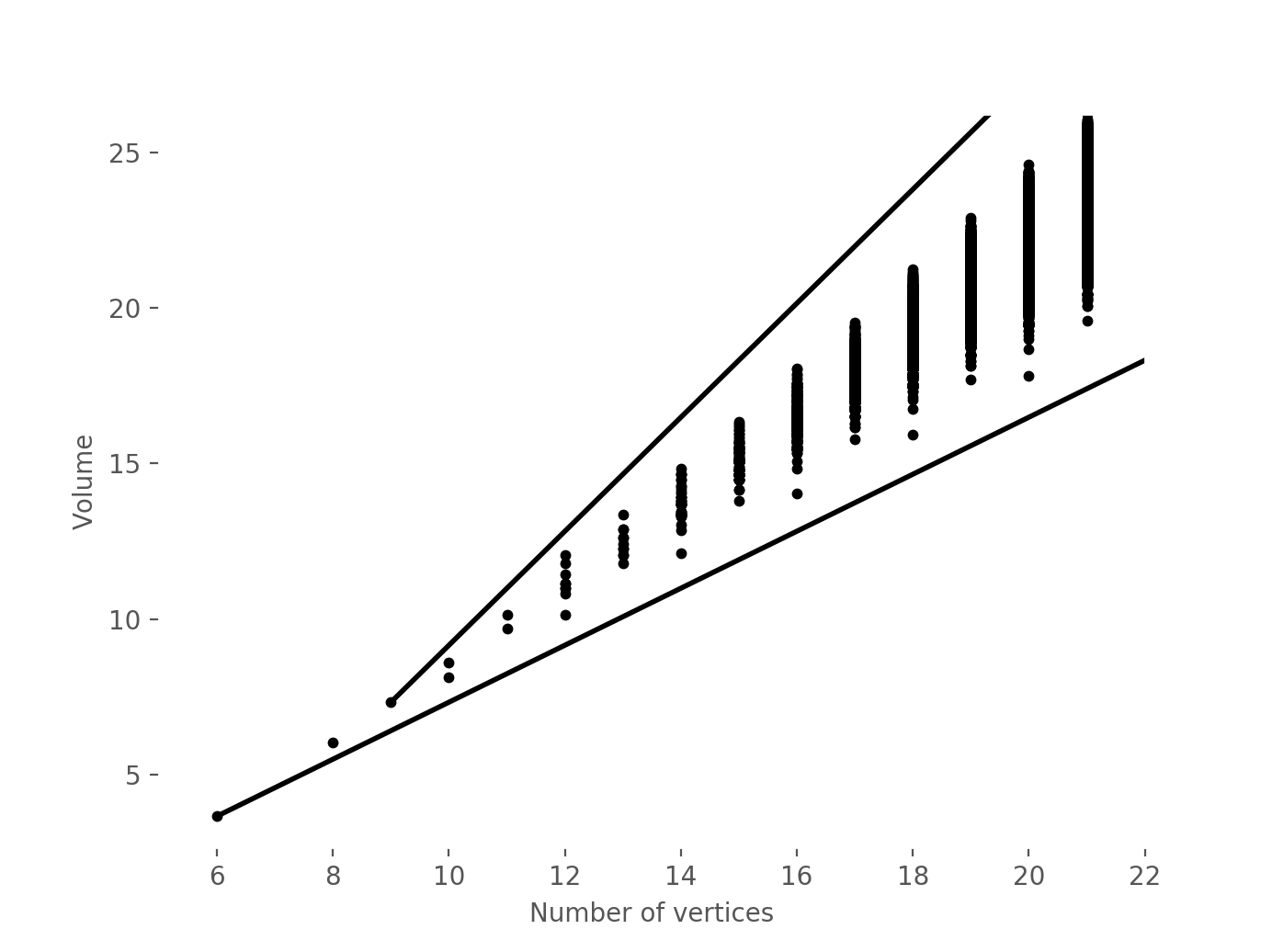}
		\caption{Volumes of ideal right-angled hyperbolic polyhedra  and volume estimates from Theorem~\ref{th101} and Theorem~\ref{th102}.} \label{fig1}
	\end{figure}
	
	Two-sided volume estimates for compact right-angled polyhedra in terms of number of vertices were obtained by Atkinson in~\cite{A09}.
	
	\begin{theorem} {\rm \cite[Theorem 2.3]{A09}} \label{th103} 
		If $P$ is a compact right-angled hyperbolic polyhedron with $V$ vertices, then 
		$$
		(V-8) \cdot \frac{v_{8}}{32} \leqslant \operatorname{vol} (P) < ( V -10) \cdot \frac{5v_{3}}{8},  
		$$
		where $v_{3}$ is the volume of a regular ideal hyperbolic tetrahedron. There is a sequence of compact polyhedra $P_{i}$, with $V_{i}$ vertices such that $\operatorname{vol} (P_{i}) / V_{i}$ approaches $5 v_{3} / 8$ as $i$ goes to infinity. 
	\end{theorem}
	
	Constant $v_{3}$ has an expression in terms of the Lobachevsky function, $v_{3} = 2 \Lambda (\pi/6)$.  
	To fifteen decimal places, $v_{3}$ is $1.014941606409653$.
	
	Recall that if $P \subset \mathbb H^{3}$ is compact right-angled, then $V \geqslant 20$ and even. The case $V = 20$ is realized for a dodecahedron, and there no compact right-angled hyperbolic polyhedra with $V=22$ vertices. 
	The upper bound from Theorem~\ref {th103} can be improved if we exclude a dodecahedron from our considerations. 
	
	\begin{theorem} \label{th104} 
		If $P$  is a compact right-angled hyperbolic polyhedron with $V \geqslant 24$ vertices, then 
		$$
		\operatorname{vol} (P) \leqslant  (V-14) \cdot \frac{5v_{3}}{8}.
		$$ 
	\end{theorem}
	
	Proof of Theorem~\ref{th104} is given in Section~\ref{sec3}. In Figure~\ref{fig2} dots presents vo\-lumes of compact right-angled hyperbolic polyhedra with at most $46$ vertices, and lines present volume estimates from Theorem~\ref{th103} (lower bound) and Theorem~\ref{th104} (upper bound). Previously, volumes of the first $825$ compact right-angled hyperbolic polyhedra were calculated in~\cite{I20} and volumes of compact right-angled polyhedra with at most $64$ vertices having only pentagonal and hexagonal faces were calculated in~\cite{EV21}.
	
	\begin{figure}[ht]
		\centering
		\includegraphics[width=1.0\textwidth]{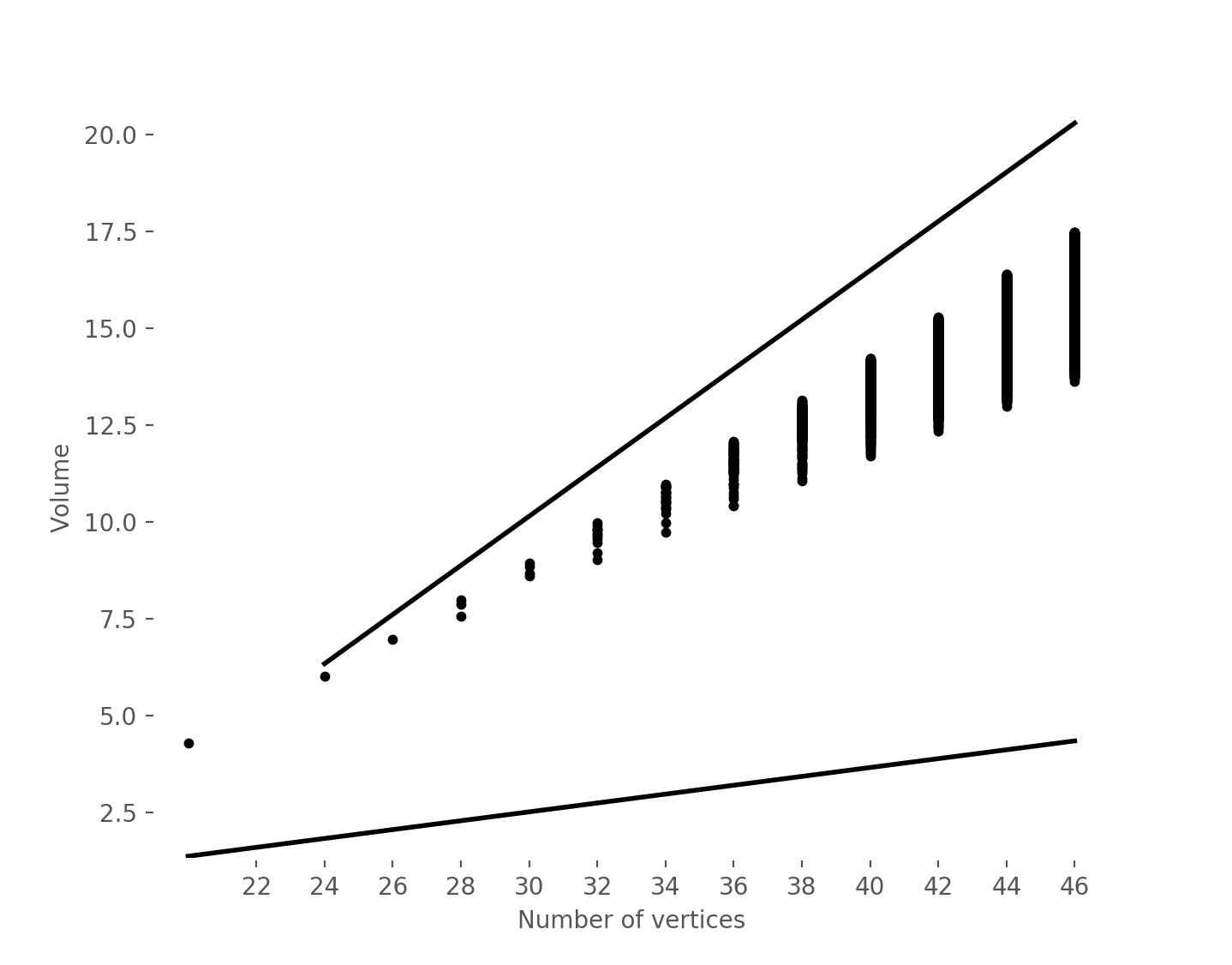}
		\caption{Volumes of compact right-angled hyperbolic polyhedra and volume estimates from Theorem~\ref{th103} and Theorem~\ref{th104}.} \label{fig2}
	\end{figure}	
	
	\section{Proof of Theorem~\ref{th102}}  \label{sec2} 
	
	We recall that 1-skeleton of any ideal right-angled hyperbolic polyhedra is 4-regular planar graph. Therefore, if $P$ is such a polyhedron with $V$ vertices and $F$ faces, then by the Euler formula $V = F-2$. 
	
	The upper estimate from Theorem~\ref {th101} was generalized in~\cite{EV20} as in Proposition~\ref{prop1} and Proposition~\ref{prop2}. 
	
	\begin{proposition} {\rm \cite[Theorem~3.2]{EV20}} \label{prop1} 
		Let $P$ be an ideal right-angled hyperbolic polyhedron with $V>6$ vertices, If $P$ has two faces $F_{1}$ and $F_ {2}$ such that $F_ {1}$ is $n_ {1}$-gonal, and $F_ {2}$ is $n_{2}$-gonal, where $n_{1}, n_{2} \geqslant 4$, then  
		$$
		\operatorname{vol} (P) \leqslant \left(V - \frac{n_{1}}{2} - \frac{n_{2}}{2} \right) \cdot \frac{v_{8}}{2}. 
		$$
	\end{proposition} 
	
	Considering $n_{1}\geqslant 5$ and $n_{2} \geqslant 5$ in Proposition~\ref{prop1}, we immediately get the following result. 
	
	\begin{corollary} \label{cor1} 
		If $P$ is an ideal right-angled hyperbolic polyhedron with $V$ vertices that has two faces such that each of them is at least $5$-gonal, then  
		$$
		\operatorname{vol} (P) \leqslant (V-5) \cdot \frac{v_{8}}{2}.
		$$ 
	\end{corollary}
	
	\begin{proposition}  {\rm \cite[Theorem~3.3]{EV20} } \label{prop2}  
		If $P$ is an ideal right-angled hyperbolic polyhedron with $V \geqslant 15$ vertices having only  triangular or quadrilateral faces, then 
		$$
		\operatorname{vol} (P) \leqslant  (V-5) \cdot \frac{v_{8}}{2}. 
		$$ 
	\end{proposition} 
	
	\begin{proposition} \label{prop3} 
		Let $P$ be an ideal right-angled hyperbolic polyhedron with $V$ vertices.  Let  $F_{1}$, $F_ {2}$ and $F_ {3}$ be three faces of $P$ such that $F_ {1}$ is $n_ {1}$-gonal, $F_ {2}$ is $n_ {2}$-gonal and $F_ {3}$ is $n_ {3}$-gonal. Assume that $F_ {2}$ is adjacent to both $F_{1}$ and $F_{3}$. Then
		$$
		\operatorname{vol} (P) \leqslant \left( V + 1 -\frac{n_{1}}{2} - \frac{n_{2}}{2} - \frac{n_{3}}{2}  \right) \cdot \frac{v_8}{2}.
		$$
	\end{proposition}
	
	\begin{proof}
		Let us consider union $P' = P \cup_{F_{2}} P$ of two copies of $P$ along $F_{2}$. Polyhedron $P'$  is ideal right-angled and has $V' = 2V - n_{2}$ vertices. Its volume is twice volume of $P$. Remark that  $P'$ has a face $F_{11}$ which is a union of two copies of $F_{1}$ along a common edge. Hence $F_{11}$ is $(2n_1-2)$-gonal. Similar, $P'$ has a face $F_{13}$ which is a union of two copies of $F_{3}$ along a common edge. Hence $F_{13}$ is $(2n_3 - 2)$-gonal.  Applying Proposition~\ref{prop1} to polyhedron $P'$ and its faces $F_{11}$ and $F_{13}$, we obtain 
		$$
		2 \operatorname{vol} (P) = \operatorname{vol} (P') \leqslant \left( (2 V - n_{2}) - \frac{2n_1-2}{2} - \frac{2n_3-2}{2} \right) \cdot \frac{v_{8}}{2}.
		$$
		Dividing both sides of the inequality by $2$, we get the result.  
	\end{proof}	
	
	\begin{corollary} \label{cor2} 
		If $P$ is an ideal right-angled hyperbolic polyhedron with $V$ vertices that has at least one $k$-gonal face, $k \geqslant 6$, then  
		$$
		\operatorname{vol} (P) \leqslant (V-5) \cdot  \frac{v_{8}}{2}. 
		$$ 
	\end{corollary}	
	
	\begin{proof}
		Consider three faces $F_{1}$, $F_{2}$ and $F_{3}$ of $P$, where one of them, say $F_{2}$, is $k$-gonal, $k \geqslant 6$, and $F_{1}$, $F_{3}$ are adjacent to $F_{2}$. Then we can apply Proposition~\ref{prop3} for the triple $(n_{1}, n_{2}, n_{3})$, where $n_{1} \geqslant 3$, $n_{2} \geqslant 6$ and $n_{3} \geqslant 3$. Then 
		\begin{eqnarray*}
			\operatorname{vol} (P) & \leqslant & \left(V + 1 -\frac{n_{1}}{2}-\frac{n_{2}}{2}-\frac{n_{3}}{2}\right) \cdot \frac{v_8}{2} \\ 
			&  \leqslant & \left(V + 1 - \frac{3}{2} - \frac{3}{2}-\frac{6}{2}\right) \cdot \frac{v_{8}}{2} = (V-5) \cdot \frac{v_{8}}{2},  
		\end{eqnarray*} 
		and the statement is proved.
	\end{proof}	
	
	\begin{proposition} \label{prop4} 
		Let $P$ be an ideal right-angled polyhedron with $V \geqslant 16$ vertices.  Assume that  $P$ has one pentagonal face and all other faces are triangles or quadrilaterals. Then 
		$$
		\operatorname{vol} (P) \leqslant (V-5) \cdot \frac{v_{8}}{2}. 
		$$ 
	\end{proposition}
	
	\begin{proof} 
		Denote by $F$ number of faces and by $p_{k}$ the number of $k$-gonal faces ($k \geqslant 3$) of polyhedron $P$. Then from $F = V+2$ we obtain 
		$$
		p_{3} = 8 + \sum_{k\geq 5} (k-4) p_{k}.
		$$ 
		By the assumption, $P$ has one pentagonal face, and therefore, nine triangular faces. Hence $F = 10 + p_{4} \geqslant 18$ and $p_{4} \geqslant 8$. 
		
		Observe  that $P$ contains three quadrilateral or pentagonal faces $F_{1}$, $F_{2}$, $F_{3}$ such that $F_{2}$ is adjacent to $F_{1}$ and $F_{3}$ both. By a contradiction, assume that any quadrilateral or pentagonal face of $P$ has at most one adjacent quadrilateral or pentagonal face. Hence any quadrilateral face has at least 3 adjacent triangular faces and the pentagonal face has at least 4 adjacent triangular faces. Therefore, totally there are at least $4 + 3 p_{4}$ sides of triangles which are adjacent to sides of a quadrilateral and pentagonal faces. Since the total number of triangular faces is $9$, we get $4 + 3 p_{4} \leqslant 27$. Since  $p_{4} \geqslant 8$, the inequality is not satisfied. 
		
		This contradiction implies that there is a triple of sequentially adjacent quadrilateral or pentagonal faces $F_{1}$, $F_{2}$, $F_{3}$, where $F_{2}$ is adjacent to $F_{1}$ and $F_{3}$. Applying Proposition~\ref{prop3} we get 
		\begin{eqnarray*}
			\operatorname{vol} (P) & \leqslant & \left(V + 1 -\frac{n_{1}}{2}-\frac{n_{2}}{2}-\frac{n_{3}}{2}\right) \cdot \frac{v_8}{2} \\ 
			& \leqslant & \left(V+ 1 - \frac{4}{2} - \frac{4}{2}-\frac{4}{2}\right) \cdot \frac{v_{8}}{2} = (V-5) \cdot \frac{v_{8}}{2}, 
		\end{eqnarray*}
		and the statement is proved.  
	\end{proof} 	
	
	%%%%%%%%%%%%%%%% 
	\begin{table}[ht] \caption{Ideal right-angled hyperbolic polyhedra.} \label{table1} 
		\begin{tabular}{|r|r|r|r|r|} \hline
			$\#$  of vertices & $\#$ of polyhedra & $\#$ of volumes & $\min$ volume & $\max$ volume \\ \hline \hline 
			6 & 1 & 1 & 3.663863  & 3.663863\\ \hline 
			7 & 0 & 0 & - & -\\ \hline
			8 & 1 & 1 & 6.023046 & 6.023046 \\ \hline 
			9 & 1 & 1 & 7.327725 & 7.327725 \\ \hline
			10 & 2 & 2 & 8.137885  & 8.612415  \\ \hline
			11 & 2 & 2 & 9.686908 & 10.149416   \\ \hline
			12 & 9 & 7 &  10.149416 & 12.046092  \\ \hline
			13 & 11 & 7 &  11.801747 & 13.350771   \\ \hline
			14 & 37 & 17 & 12.106298 & 14.832681   \\ \hline
			15 & 79 & 31 & 13.813278 & 16.331571   \\ \hline
		\end{tabular}
	\end{table}
	
	Summarizing Proposition~\ref{prop2} (the case if there is no $k$-gonal faces for $k \geqslant 5$), Corollary~\ref{cor2} (the case if there is at least one $k$-gonal face, $k \geqslant 6$), Proposition~\ref{prop4} (the case if there is one pentagonal face and other faces are $k$-gonal, $k \leqslant 4$), and Corollary~\ref{cor1} (the cases if there are two faces such that each of them has at least $5$ edges) we have the statement for polyhedra with $V \geqslant 16$ vertices.  For polyhedra $P$ with $9 \leqslant V \leqslant 15$ vertices the inequality 
	$\operatorname{vol} (P) \leqslant (V - 5) \cdot \frac{v_{8}}{2}$
	holds by direct computation~\cite{EV20}, see Table~\ref{table1}. Proof of Theorem~\ref{th102} is completed. 
	
	\section{Proof of Theorem~\ref{th104}}  \label{sec3}
	
	The proof of Theorem~\ref{th104} uses the same method as the proof of Theorem~\ref{th102} presented in previous section. But with the difference that instead of ideal right-angled polyhedra with 4-regular 1-skeletons we consider compact right-angled hyperbolic polyhedra with 3-regular 1-skeletons. Therefore, if $P$ is a compact right-angled hyperbolic polyhedron with $V$ vertices and $F$ faces, then by the Euler formula $2F = V + 4$. 
	
	The upper estimate from Theorem~\ref {th102} was generalized in~\cite{EV21} as in Proposition~\ref{prop5}, Proposition~\ref{prop6} and Proposition~\ref{prop7} below.  For the reader's convenience we present these statements with proofs. 
	
	\begin{proposition} {\rm \cite[Theorem~3.4]{EV21}} \label{prop5} 
		Let $P$ be a compact right-angled hyperbolic polyhedron with $V$ vertices. Let $F_{1}$ and $F_ {2}$ be two faces of $P$ such that $F_ {1}$ is $n_ {1}$-gonal, and $F_ {2}$ is $n_{2}$-gonal, where $n_{1}, n_{2} \geqslant 6$. Then 
		$$
		\operatorname{vol} (P) \leqslant \left( V - n_{1} - n_{2} \right) \cdot \frac{5v_3}{8}.
		$$
	\end{proposition} 
	\begin{proof}
		Case (1). Let us assume that faces $F_ {1}$ and $F_ {2}$ are not adjacent. We construct a family $\{P_{i}\}$ of right-angled polyhedra by induction, attaching on each step a~copy of the polyhedron $P$. Put $P_ {1} = P$. Define $P_{2}=P_ {1} \cup_{F_{1}}P_{1}$ by identifying two copies of the polyhedron $P_ {1}$ along the face $F_ {1}$. It has $V_ {2} = 2V-2n_ {1}$ vertices. Indeed, faces of $P_{1}$ adjacent to $F_{1}$ form dihedral angles $\pi/2$ with $F_{1}$, so vertices of $F_{1}$ aren't vertices in the glued polyhedron $P_{2}=P_ {1} \cup_{F_{1}}P_{1}$. For the volumes we have $\operatorname{vol} (P_{2}) = 2 \operatorname{vol} (P)$. The polyhedron $P_{2}$ has at the least one face isometric to $F_{2}$. Attaching the polyhedron $P$ to the polyhedron $P_{2}$ along this face we get  $P_{3} = P_{2} \cup_{F_{2}} P = P \cup_{F_{1}} P \cup_{F_{2}} P$. Evidently, $P_{3}$ is a compact right-angled polyhedron with $V_{3} = 3V - 2n_{1} - 2n_{2}$ vertices and volume $\operatorname{vol} (P_{3}) = 3 \operatorname{vol} (P)$. Continuing the process of attaching  $P$ alternately through the faces isometric to $F_{1}$ and $F_{2}$, we get the polyhedron $P_{2k+1} = P_{2k-1} \cup_{F_{1}} P \cup_{F_{2}} P$, which is a compact right-angled polyhedron with $V_{2k+1} = (2k+1) V - 2 k n_{1} - 2 k n_{2}$ vertices and volume $\operatorname{vol} (P_{2k+1}) = (2k+1) \operatorname{vol} (P)$. Now let us apply the upper bound from Theorem~\ref{th103} to polyhedron $P_{2k+1}$:
		$$
		(2k+1) \operatorname{vol} (P) < \left( (2k+1) V - 2 k n_{1} - 2 k n_{2} - 10 \right) \frac{5 v_{3}}{8}.  
		$$
		Dividing both sides of the inequality by $(2k+1)$ and passing to the limit as $k \to \infty $, we obtain the required inequality.
		
		Case (2). Let us assume that faces $F_{1}$ and $F_{2}$ are adjacent. Put $P' =P \cup_{F_{1}} P$. Then polyhedron $P'$ has $V'=2V-2n_{1}$ vertices and $\operatorname{vol} (P') = 2 \operatorname{vol} (P)$. By construction, the polyhedron $P'$ has a face $F'_{2}$, which is a union of two copies of $F_{2}$ along a common edge. Hence, $F'_{2}$ is a $(2n_{2}-4)$-gon. Since the face $F_1$ has at least $5$ sides, there is a face $F_{3}$ in $P$ adjacent to $F_{1}$, but not adjacent to $F_{2}$. As a result of attaching $P$ along $F_{1}$, the face $F_{3}$ will turn into a face $F'_{3}$ in $P'$ that has at least $6$ sides.  Thus, in $P'$ there is a pair of non-adjacent faces $F'_{2}$ and $F'_{3} $, each of which has at least $6$ sides. Applying case  (1)  for the polyhedron $P'$ and its non-adjacent faces $F'_{2}$ and $F'_{3}$ we get:
		$$
		2 \operatorname{vol} (P) \leqslant \left( V' - (2n_{2} - 4) - 6 \right) \frac{5 v_{3}}{8} =  \left( 2V-2n_{1}-2n_{2} - 2 \right) \frac{5 v_{3}}{8}
		$$
		and hence, 
		$ \displaystyle \operatorname{vol} (P) < \left( V-n_{1}-n_2 \right) \frac{5 v_{3}}{8}$. 
	\end{proof}

	Considering $n_{1}\geqslant 7$ and $n_{2} \geqslant 7$ in Proposition~\ref{prop5} we immediately get the following result. 
	
		\begin{corollary} \label{cor3} 
		If $P$ is a compact right-angled hyperbolic polyhedron with $V$ vertices having two faces such that each of them is at least $7$-gonal. Then 
		$$
		\operatorname{vol} (P) \leqslant (V-14) \cdot \frac{5v_{3}}{8}.
		$$ 
	\end{corollary}

	\begin{proposition} {\rm \cite[Corollary~3.2]{EV21} } \label{prop6} 
		Let $P$ be a compact right-angled hyperbolic polyhedron with $V$ vertices. Let $F_{1}$, $F_ {2}$ and $F_ {3}$ be three faces of $P$ such that $F_ {1}$ is $n_ {1}$-gonal, $F_ {2}$ is $n_ {2}$-gonal and $F_ {3}$ is $n_ {3}$-gonal.  Assume that $F_ {2}$ is adjacent to both $F_{1}$ and $F_{3}$. Then
		$$
		\operatorname{vol} (P) \leqslant \left(V - n_{1} - n_{2} - n_{3} + 4  \right) \cdot \frac{5v_3}{8}.
		$$
	\end{proposition}
	\begin{proof}
		Let us consider polyhedron $P'=P \cup_{F_{2}} P$. It has $V'=2V-2n_{2}$ vertices and $\operatorname{vol} (P') = 2 \operatorname{vol} (P)$. Remark that  $P'$ has a face $F'_{1}$ which is a union of two copies of $F_{1}$ and has $2n_1-4$ vertices. Similar, $P'$ has a face $F'_{3}$ which is a union of two copies of $F_{3}$ and has $2n_3-4$ vertices.  Applying Proposition~\ref{prop5} to polyhedron $P'$ and its faces $F'_{1}$ and $F'_{3}$. We get
		$$
		2\operatorname{vol} (P) = \operatorname{vol} (P') \leqslant (V' - (2n_1-4)-(2n_3-4)) \cdot \frac{5v_3}{8}.
		$$
		Using formula for $V'$ and dividing both sides of the inequality by $2$ we get the result.  
	\end{proof}

	\begin{proposition} {\rm \cite[Theorem~3.5]{EV21} }  \label{prop7}  
		If $P$ is a compact  right-angled hyperbolic polyhedron with $V \geqslant 46$ vertices and with only pentagonal or hexagonal faces, then 
		$$
		\operatorname{vol} (P) \leqslant  (V-14) \cdot \frac{5v_{3}}{8}. 
		$$ 
	\end{proposition} 
	
	\begin{proof}
			Denote by $F$ number of faces and by $p_{k}$ the number of $k$--gonal faces ($k \geqslant 5$) of polyhedron $P$. Then from $2F = V+4$ we obtain 
		$$
		p_{5} = 12 + \sum_{k\geqslant 7} (k-6) p_{k}.
		$$ 
		By the assumption, $P$ has only pentagonal and hexagonal faces. Hence $F = 12 + p_{6} \geqslant 25$ and $p_{6} \geqslant 13$. 
	We observe that in the polyhedron $P$ there are three hexagonal faces $F_{1}$, $F_{2}$, $F_{3}$ such that $F_{2}$ is adjacent to both $F_{1}$ and $F_{3}$. Assume by contradiction that there is no such triple of faces. Then each hexagonal face is adjacent to at most one hexagonal face. If a hexagonal face has no adjacent hexagons (we will say that it is isolated), then it is adjacent to $6$ pentagonal faces. If two hexagons are adjacent to each other and none of them is adjacent to another hexagon (we will say that the faces form a pair), then their union is adjacent to $10$  pentagonal faces. If we have $k_{1}$ isolated hexagonal faces and $k_{2}$ pairs of hexagonal faces, then they are adjacent to pentagonal faces through $6 k_{1} + 10 k_{2}$ sides. Since $p_{5} = 12$ we have $6k_{1} + 10 k_{2} \leqslant  60$.  But  $k_{1} + 2 k_{2} = p_{6 }\geqslant  13$ implies $2k_{2} \geqslant 13 - k_{1}$, hence $6k_{1} + 10 k_{2} \geqslant 65 + k_{1} > 60$. This contradiction implies that there is a triple of hexagonal faces $F_{1}$, $F_{2}$, $F_{3}$, where $F_{2}$ is adjacent to $F_{1}$ and $F_{3}$. By applying Proposition~\ref{prop6}  we get the result. 
	\end{proof}

	\begin{corollary} \label{cor4} 
		If $P$ is a compact right-angled hyperbolic polyhedron with  $V$ vertices and at least one $k$-gonal face, $k \geqslant 8$. Then 
		$$
		\operatorname{vol} (P) \leqslant (V - 14) \cdot  \frac{5v_{3}}{8}. 
		$$ 
	\end{corollary}	
	
	\begin{proof}
		Consider three faces $F_{1}$, $F_{2}$ and $F_{3}$ of $P$, where one of them, say $F_{2}$, is $k$-gonal, $k \geqslant 8$, and $F_{1}$, $F_{3}$ are adjacent to $F_{2}$. Then we can apply Proposition~\ref{prop6} for the triple $(n_{1}, n_{2}, n_{3})$, where $n_{1} \geqslant 5$, $n_{2} \geqslant 8$ and $n_{3} \geqslant 5$. Then 
		\begin{eqnarray*}
			\operatorname{vol} (P) & \leqslant & \left(V  - n_{1} - n_{2} - n_{3} + 4 \right) \cdot \frac{5v_3}{8} \\ 
			& \leqslant & \left(V - 5 - 8 - 5 + 4 \right) \cdot \frac{5v_{3}}{8} = (V-14) \cdot \frac{5v_{3}}{8}, 
		\end{eqnarray*} 
		that gives the result. 
	\end{proof}	
	
	\begin{proposition} \label{prop8} 
		Let $P$ be a compact right-angled polyhedron with $V \geqslant 48$ vertices. Assume that $P$ has one heptagonal face and all other faces are pentagons or hexagons. Then 
		$$
		\operatorname{vol} (P) \leqslant (V - 14) \cdot \frac{5v_{3}}{8}. 
		$$ 
	\end{proposition}
	
	\begin{proof}
		Denote by $F$ number of faces and by $p_{k}$ the number of $k$--gonal faces ($k \geqslant 5$) of polyhedron $P$. Then from $2F = V+4$ we obtain 
		$$
		p_{5} = 12 + \sum_{k\geqslant 7} (k-6) p_{k}.
		$$ 
		By the assumption, $P$ has one heptagonal face, and therefore, $13$ pentagonal faces. Hence $F = 14 + p_{6} \geqslant 26$ and $p_{6} \geqslant 12$. 
		
		Observe that $P$ contains three hexagonal or heptagonal faces $F_{1}$, $F_{2}$, $F_{3}$ such that $F_{2}$ is adjacent to $F_{1}$ and $F_{3}$ both. By a contradiction, assume that any hexagonal or heptagonal face of $P$ has at most one adjacent hexagonal or heptagonal face. Hence any hexagonal face has at least $5$ adjacent pentagonal faces and the heptagonal face has at least $6$ adjacent pentagonal faces. Therefore, totally there are at least $6 + 5 p_{6}$ sides of pentagons which are adjacent to sides of hexagonal and heptagonal faces. Since the total number of pentagonal faces is $13$, we get $6 + 5 p_{6} \leqslant 65$. Since  $p_{6} \geqslant 12$, the inequality is not satisfied. 
		
		This contradiction implies that there is a triple of sequentially adjacent hexagonal or heptagonal faces $F_{1}$, $F_{2}$, $F_{3}$, where $F_{2}$ is adjacent to $F_{1}$ and $F_{3}$. Applying Proposition~\ref{prop7} we get 
		\begin{eqnarray*} 
			\operatorname{vol} (P) & \leqslant & \left(V  - n_{1} - n_{2} - n_{3} + 4 \right) \cdot \frac{5v_3}{8} \\ 
			& \leqslant & \left(V  - 6 -  6 - 6 + 4 \right) \cdot \frac{5v_3}{8}  = (V-14) \cdot \frac{5v_{3}}{8}, 
		\end{eqnarray*} 
		that completes the proof. 
	\end{proof} 
	
	\begin{table}[ht] \caption{Compact right-angled hyperbolic polyhedra.} \label{table2} 
		\begin{tabular}{|r|r|r|r|r|} \hline
			$\#$  of vertices & $\#$ of polyhedra & $\#$ of volumes & $\min$ volume & $\max$ volume \\ \hline \hline 
			20 & 1 & 1 & 4.306208  & 4.306208 \\ \hline 
			22 & 0 & 0 & - & -  \\ \hline 
			24 & 1 & 1 & 6.023046  & 6.023046 \\ \hline 
			26 & 1 & 1 & 6.967011 & 6.967011\\ \hline
			28 & 3 & 3 & 7.563249  &  8.000234\\ \hline
			30 & 4 & 4 & 8.612415   & 8.946606 \\ \hline
			32 & 12 & 12 & 9.019053  & 9.977170\\ \hline
			34 & 23 & 23 & 9.730847 &10.986057  \\ \hline
			36 & 71 & 71 & 10.416044 & 12.084191 \\ \hline
			38 & 187 & 187 &  11.058763  & 13.138893 \\ \hline
			40 & 627 & 627 & 11.708462 &  14.222648 \\ \hline
			42 & 1970 & 1952 & 12.352835  & 15.300168  \\ \hline
			44 & 6833 & 6771 &  12.996118 & 16.397833 \\ \hline
			46 & 23384 & 23082 & 13.637792 & 17.486616 \\ \hline
		\end{tabular}
	\end{table}	
	
	Summarizing Proposition~\ref{prop7} (the case if there is no $k$-gonal faces for $k \geqslant 7$), Corollary~\ref{cor4} (the case if there is at least one $k$-gonal face, $k \geqslant 8$), Proposition~\ref{prop8} (the case if there is one heptagonal face and other faces are $k$-gonal, $k \leqslant 6$), and  Corollary~\ref{cor3} (the cases if there are two faces such that each of them has at least $7$ edges) we have the statement for polyhedra with $V \geqslant 48$ faces.  For polyhedra with $24 \leqslant V \leqslant 46$ vertices the inequality holds by direct computation, see Table~\ref{table2}. This completes the proof of Theorem~\ref{th104}.

\end{document}